\documentclass[11pt,reqno,oneside]{amsart}
\usepackage{graphicx} 
\usepackage[margin=1in]{geometry}
\usepackage{amsmath}
\usepackage{amsthm}
\usepackage{bm}
\usepackage[sort]{natbib}

\usepackage{algpseudocode}
\usepackage{algorithm}
\usepackage{subfigure}
\usepackage{blkarray}
\usepackage{mathtools}
\usepackage{xcolor}
\usepackage{hyperref}
\usepackage{cleveref}
\usepackage{tikz}
\usetikzlibrary{3d}
\usepackage{tikz-3dplot}
\usepackage{pgfplots}
\usepackage{textcomp}

\theoremstyle{definition}
\newtheorem{theorem}{Theorem}[section]

\newtheorem{remark}[theorem]{Remark}
\newtheorem{corollary}[theorem]{Corollary}
\newtheorem{example}[theorem]{Example}
\let\oldtextit\textit 
\renewcommand\emph[1]{\oldtextit{\color{RoyalBlue}#1}}
\definecolor{RoyalBlue}{cmyk}{1, 0.50, 0, 0}

\crefname{remark}{remark}{remarks}
\Crefname{remark}{Remark}{Remarks}

\title[Certified surface approximations]{Certified surface approximations using~the~interval~Krawczyk~test}

\begin{document}


\begin{abstract}
We propose an algorithm to construct a certified approximation of a surface by generalizing the Krawczyk test. The Krawczyk test is based on interval arithmetic, and confirms the existence and uniqueness of a solution to a square system of analytic equations in a region.  By generalizing this test, we extend the reach of this technique to non-square systems and higher-dimensional varieties.  We provide a prototype implementation and illustrate its use on several examples.
\end{abstract}

\author{Michael Burr}
\address{Clemson University, Clemson, SC, USA}
\email{burr2@clemson.edu}

\author{Jonathan D. Hauenstein}
\address{University of Notre Dame, Notre Dame, IN, USA}
\email{hauenstein@nd.edu}

\author{Kisun Lee}
\address{Clemson University, Clemson, SC, USA}
\email{kisunl@clemson.edu}

\maketitle

\section{Introduction}

A standard problem in computational algebraic
geometry is to compute and analyze the solution
set to a system of polynomial equations.  
From a numerical perspective, one aims to 
compute numerical approximations with certificates, i.e., 
numerical proofs of the existence and uniqueness of a nearby solution.  For example, Smale's $\alpha$-theory~\cite{SmaleAlphaTheory}
analyzes Newton's method using data gathered at 
one point to provide such a certificate.  This technique has been used to develop certified homotopy
continuation trackers, e.g., see~\cite{beltran2012certified,Smale:1993,Hauenstein:2016}.  

Rather than using data from a single point, 
an alternative is to apply interval arithmetic
to understand the behavior in a region. 
The Krawczyk operator was originally introduced by 
Krawczyk~\cite{krawczyk1969newton} as a way to refine a compact region to identify solutions to systems
of equations using a Newton-like operator.
In~\cite{moore2009introduction}, Moore showed that this operator can be used to certify the existence of a solution.  Rump~\cite{rump1983solving} built on these results to show that this operator can also be used to conclude the uniqueness of a solution in a region.

Recently, the Krawczyk test was adapted and used in \cite{duff2024certified,guillemot2024validated,burr2025certified} to create certified algorithms for homotopy continuation and curve tracking.  The key insight in \cite{duff2024certified,guillemot2024validated} was that 
the Krawczyk test could be applied simultaneously 
over an interval to obtain certificates over an entire region.
In particular, such a test was used to certify the 
existence and uniqueness of a curve segment 
traveling the length of a box.  

The main theoretical contribution provided in this paper 
is that one can simultaneously apply the Krawczyk test over 
every point of a region
in any dimension for well-constrained systems.
In particular, this generalized test permits
certified approximations of surfaces, which
are the main objects of interest here.  

The rest of the paper is organized as follows.
Section~\ref{sec:GenKrawczyk} considers an interval
Krawczyk test applied to well-constrained systems
with Theorem~\ref{thm:generalized_krawczyk} providing
the main theoretical contribution.
This test is considered for 
single-sheeted and multi-sheeted graphs
in Section~\ref{sec:InitialCaseGraphs}.
Limitations of the first approach 
are reviewed in Section~\ref{sec:limitations}
along with improvements which address those limitations.
Section~\ref{sec:MainAlg} describes
the complete certified surface approximation algorithm.
Finally, Section~\ref{sec:Implementation},
considers implementation details and examples.

\section{Interval Krawczyk test}\label{sec:GenKrawczyk}
The \emph{Krawczyk test} is a commonly used tool in certified numerical algorithms as it provides the existence and uniqueness of a solution to a square analytic system within a region.  In particular, the test checks whether a Newton-like operator is contractive on a region.
Our first goal is to extend this test 
so that it can be applied to higher-dimensional varieties.  We then apply this generalization of the Krawczyk test to the problem of approximating surfaces.

\subsection{Interval arithmetic}
The Krawczyk test is based on \emph{interval arithmetic}, which is a way to compute with intervals instead of points.  Two advantages of using intervals are: (1) we can perform conservative calculations when exact values are unknown or cannot be represented on a computer and (2) we can perform calculations over and make conclusions about an entire region at once. 

Operations on intervals are defined entry-wise.
That is, for an arithmetic operator $\odot$ and intervals $I$ and $J$, we define
$$
I\odot J\vcentcolon=\{x\odot y:x\in I,y\in J\}.
$$
For example, $[a,b]+[c,d]=[a+c,b+d]$, see \cite{moore2009introduction} for more details.  Let $\mathbb{I}\mathbb{Q}^m$ be the set of axis-aligned $m$-dimensional boxes whose vertices are in $\mathbb{Q}^m$.  We may identify $(\mathbb{IQ})^m$ with $\mathbb{I}\mathbb{Q}^m$, and interchangeably use the notation $(I_1,\dots,I_m)$ or $I_1\times\cdots\times I_m$.

For a function $f:\mathbb{R}^m\rightarrow\mathbb{R}$, 
an \emph{interval extension} of $f$ is any function $\square f:\mathbb{I}\mathbb{Q}^m\rightarrow\mathbb{I}\mathbb{Q}$
where, for all $I\in\mathbb{I}\mathbb{Q}^m$, 
$$
\square f(I)\supseteq f(I)\vcentcolon=\{f(x)\mid x\in I\}.
$$
An interval extension for a system of polynomials consists of a system of extensions, one for each polynomial.

For an interval $I\in\mathbb{I}\mathbb{Q}$, we define $\|I\|=\max_{x\in I}|x|$ to be the maximum of the elements of $I$.  For a box $I=I_1\times\dots\times I_m\in\mathbb{I}\mathbb{Q}^m$, we define $\|I\|=\max_i \|I_i\|.$  

\subsection{Interval Krawczyk test}\label{sec:IntervalKrawczyk}

Let $F=(f_1,\dots,f_{n-d})$ be a system of polynomials, and suppose that their set of common zeros $X\vcentcolon=\mathcal{V}(f_1,\dots,f_{n-d})$ is a smooth pure $d$-dimensional variety in $\mathbb{R}^n$.  Let $\pi_d:\mathbb{R}^n\rightarrow\mathbb{R}^d$ be the projection onto the first  $d$ coordinates and, similarly, $\pi_{-d}$ the projection onto the last $n-d$ coordinates.  Suppose that $\pi_d|_X:X\rightarrow\mathbb{R}^d$ contains a $d$-dimensional set $U$ with nonempty interior in its image, then, for most $x\in U$, there are only finitely many $z\in X$ projecting to $x$.  We now define a test, based on the Krawczyk test, so that if $I=I_1\times\dots\times I_n\in\mathbb{I}\mathbb{Q}^n$ passes the test, then we know for each $x\in I_1\times\cdots \times I_d$, there is a unique $y\in I_{d+1}\times \cdots\times I_n$ such that~$(x,y)\in X$.

Let $\hat{z}\in\mathbb{R}^n$, $r=(r_1,r_2)$ a pair of positive numbers, and $A$ an invertible matrix.  We define the \emph{interval Krawczyk operator} as
$$
K(F,\hat{z},r,A)\vcentcolon= -A\square F(I,\pi_{-d}(\hat{z}))+(Id_{n-d}-A\square J_F(I,J))(J-\pi_{-d}(\hat{z})),
$$
where 
\mbox{$I\vcentcolon=\pi_d(\hat{z})+r_1[-1,1]^d\subseteq\mathbb{R}^d$}, 
$Id_{n-d}$ is the identity matrix of size $(n-d)\times (n-d)$, 
$J\vcentcolon=\pi_{-d}(\hat{z})+r_2[-1,1]^{n-d}\subseteq\mathbb{R}^{n-d}$, 
and $\square J_F(I,J)$ is the interval extension of the following 
$(n-d)\times(n-d)$ Jacobian submatrix, evaluated over $I\times J$
\[
\begin{bmatrix}
\frac{\partial f_i}{\partial z_j}
\end{bmatrix}_{\substack{i\in\{1,\dots,n-d\}\\j\in\{d+1,\dots,n\}}}.
\]

\begin{theorem}[Interval Krawczyk test]\label{thm:generalized_krawczyk}
Suppose $F$, $\hat{z}$, $r=(r_1,r_2)$, $A$, $I$, and $J$ are defined as above.  If there exists $\rho\in(0,1)$ with
\[
K(F,\hat{z},r,A)\subset \rho\,(J-\pi_{-d}(\hat{z})),
\]
then, for every $\hat{x}\in I$, there exists a unique
$y^*\in J$ such that $F(\hat{x},y^*)=0$
with $\|y^*-\pi_{-d}(\hat{z})\|\le r_2\rho$. In this case, we say that $\hat{z}$ is a \emph{$\rho$-approximate solution} to $F$ with \emph{certification~radius}~$r$.
\end{theorem}
\begin{proof}
Let $\hat{z}=(\hat{x},\hat{y})\in\mathbb{R}^n$ and define $G(y)\vcentcolon=F(\hat{x},y)$.  
Since interval arithmetic is inclusion isotone, it immediately
follows that \mbox{$K(G,\hat{y},r,A)\subset K(F,\hat{z},r,A)$}, where
\[K(G,\hat{y},r,A)\vcentcolon=-AG(\hat{y})+(Id_{n-d}-A\square J_G(J))(J-\hat{y}).\]
When the containment assumption of the theorem holds, we know 
that $K(G,\hat{y},r,A)\subset \rho(J-\pi_{-d}(\hat{z}))=\rho(J-\hat{y})$. Then, by the standard Krawczyk test \cite{moore2009introduction,duff2024certified,guillemot2024validated}, the existence and uniqueness of a solution to~$G$ follows. Since $\hat{x}$ is arbitrary from $I$, the result follows.
\end{proof}

\begin{remark}\label{rem:Achoice}
In many applications, a good choice of $A$ is to approximate the inverse of a matrix in $\square J_F(I,J)$.  Therefore, a natural choice for $A$ is to approximate
\[
\begin{bmatrix}
\frac{\partial f_i}{\partial z_j}(\hat{z})
\end{bmatrix}^{-1}_{\substack{i\in\{1,\dots,n-d\}\\j\in\{d+1,\dots,n\}}}
\]
\end{remark}

We turn \Cref{thm:generalized_krawczyk} into the \emph{IntervalKrawczykTest} in Algorithm~\ref{algo:Krawczyk-test}.  Since both $K(F,\hat{z},r,A)$ and $J-\pi_{-d}(\hat{z})$ are intervals centered at zero and $\rho(J-\pi_{-d}(\hat{z}))=[-\rho,\rho]^{n-d}$, the final test of Algorithm~\ref{algo:Krawczyk-test} is equivalent to the inclusion in \Cref{thm:generalized_krawczyk}.

\algrenewcommand\algorithmicrequire{\textbf{Input}:}
\algrenewcommand\algorithmicensure{\textbf{Output}:}

\begin{algorithm}[ht]
	\caption{IntervalKrawczykTest}
 \label{algo:Krawczyk-test}
\begin{algorithmic}[1]
\Require  
\begin{itemize}
    \item A polynomial system $F=(f_1,\dots, f_{n-d}):\mathbb{R}^n\rightarrow \mathbb{R}^{n-d}$,
    \item a point $z=(z_1,\dots, z_n)\in \mathbb{Q}^n$,
    \item a vector of positive rational numbers $r\in \mathbb{Q}^2$,
    \item an invertible matrix $A\in\mathbb{Q}^{(n-d)\times (n-d)}$, and 
    \item $\rho\in (0,1)$.
\end{itemize}
\Ensure A boolean, the value of interval Krawczyk test.
\State {Set $I=\pi_d(z)+r_1[-1,1]^d$ and $J=\pi_{-d}(z)+r_2[-1,1]^{n-d}$.}
\State {Set $K=-A\square F(I,\pi_{-d}(z))+(Id_{n-d}-A\square J_F(I,J))(J-\pi_{-d}(z))$.}
\State {\Return $\|K\|< r_2\rho$.}
\end{algorithmic}
\end{algorithm}

\begin{remark}
    Although the focus is on real varieties in this paper, it is straight-forward to extend the results to complex varieties by considering $\mathbb{C}^n$ as $\mathbb{R}^{2n}$, and modifying the definitions appropriately.  In addition, we may relax the condition that $X$ is smooth, since it only needs to be smooth where the test is applied.
\end{remark}

\section{An initial case: graphs}\label{sec:InitialCaseGraphs}

As a first step to show the power of this interval Krawczyk test, we consider the case 
when $X$ is a graph.  The goal of this section is to provide a simple algorithm that correctly approximates the graph $X$.  It provides a motivating approach whose drawbacks are addressed in subsequent sections.  Since the theory is identical in both the surface and higher-dimensional cases, we present the material in full generality even though we reduce to the surface case in subsequent sections.  We also note that we do not address the question of deciding whether a variety is a graph over some domain.

Suppose that $X$ is defined as in \Cref{sec:IntervalKrawczyk} and $U\subseteq\mathbb{R}^d$.  We say that $X$ is a \emph{graph} over $U$ if for every $x\in U$, the fiber $X\cap\pi_d^{-1}(x)$ has exactly one point.  In addition, we say that $X$ is \emph{regular} over $U$ if for all $z\in X\cap\pi_d^{-1}(U)$, the Jacobian \[
\begin{bmatrix}
\frac{\partial f_i}{\partial z_j}(z)
\end{bmatrix}_{\substack{i\in\{1,\dots,n-d\}\\j\in\{d+1,\dots,n\}}}
\]
is nonsingular.  This implies, in particular, that the tangent space of $X$ at $z$ does not include any nonzero vectors parallel to the direction of projection.  In this case, by the implicit function theorem, $X\cap\pi_d^{-1}(U)$ is homeomorphic to $U$.  Our goal in this section is to build an enclosure of $X\cap\pi_d^{-1}(U)$ that is arbitrarily close in Hausdorff distance and deformation retracts to $X\cap\pi_d^{-1}(U)$.

\subsection{Graph approximations}
We begin by specializing our problem.  Suppose that each defining polynomial has rational coefficients, i.e., $f_i\in\mathbb{Q}[z_1,\dots,z_n]$, $U$ is a box in $\mathbb{I}\mathbb{Q}^d$, and $X$ is a graph over $U$.  Moreover, we fix a value for $\rho\in(0,1)$.  By scaling variables, we immediately reduce to the case where $U$ is a cube, i.e., all the sides of $U$ are the same length.

We follow a subdivision-based approach on $U$, where a subregion is processed and either accepted as containing the surface or rejected and further refined.  Suppose that $B\subseteq U$ is a cube in $\mathbb{I}\mathbb{Q}^d$ of side length $2r_1$, then the center $\hat{x}$ of the box is in $\mathbb{Q}$, and we can use one of many certified multivariate real root solvers, see, e.g., \cite{Yap:MirandaTest:2020,Cheng:CertifiedRealIsolation:2022}, to approximate the solution $\hat{y}$ in the fiber to high precision.  Then $\hat{z}=(\hat{x},\hat{y})$ is a point near $X$, and we use this point in the interval Krawczyk test, \Cref{thm:generalized_krawczyk}.  We choose $A$ as the inverse of the Jacobian at $\hat{z}$, as in Remark \ref{rem:Achoice}.  Thus, we only need a value for $r_2$ to apply the Krawczyk test.

The choice of $r_2$ is more subtle than the remaining parameters.  When performing the interval Krawczyk test, the base region $B$ is given (and takes the role of $I$ in \Cref{thm:generalized_krawczyk}).  On the other hand, $J$ is built from $r_2$.  In order for the interval Krawczyk test to succeed, it is necessary (but not sufficient) for $B\times J$ to contain $\pi_d^{-1}(B)$.  From our setup, there are no explicit bounds given on the magnitude of the slopes appearing in $X$.  While there is an upper bound due to compactness, we do not know this bound {\em a priori}.  Therefore, during the subdivision process, we choose $r_2$ so that it approaches zero at a slower rate than $r_1$ under subdivision.  With such a choice, the box $B\times J$ eventually contains $\pi_d^{-1}(B)$, and, moreover, we can prove that the algorithm terminates.

We combine these choices into the \emph{GraphApproximation} algorithm, Algorithm \ref{algo:Graph-approx}.  In this algorithm, upon failure of the Krawczyk test on box $B$, $B$ is split into $2^d$ sub-boxes formed by splitting each of the sides of $B$ in half.

\begin{algorithm}[ht]
	\caption{GraphApproximation}
 \label{algo:Graph-approx}
\begin{algorithmic}[1]
\Require  
\begin{itemize}
    \item A polynomial system $F=(f_1,\dots, f_{n-d}):\mathbb{R}^n\rightarrow \mathbb{R}^{n-d}$,
    \item a box $U\in\mathbb{I}\mathbb{Q}^d$ such that $\mathcal{V}(F)$ is a graph over $U$,
    \item $\rho\in(0,1)$.
\end{itemize}
\Ensure A collection of boxes covering the graph of $F$ over $U$.
\State{Apply a linear scaling so that $U$ is a cube of side length $2$}
\State{Initialize queue $\mathcal{Q}$ to contain $U$}
\While{$\mathcal{Q}$ is not empty}
\State{Pop $B$ from $\mathcal{Q}$}
\State{Set $2^{s_1+1}$ to be the side length of $B$}
\State{Set $s_2=\lceil s_1/2\rceil$}
\State{Set $\hat{x}$ to be the midpoint of $B$}
\State{Set $\hat{y}$ to be an approximation of the zero of $F$ over $\hat{x}$ of accuracy at least $2^{s_2}$}\label{hat_y_approx}
\State {Set $\hat{z}=(\hat{x},\hat{y})$}
\State{Set $\hat{A}$ to be an approximation of the inverse Jacobian at $\hat{z}$}\label{hat_A_approx}
\If{IntervalKrawczykTest$(F,\hat{z},(2^{s_1},2^{s_2}),\hat{A},\rho)=\textsc{True}$}\label{line:intervalKrawczyk}
\State{{\bf report} $B\times (\hat{y}+2^{s_2}[-1,1]^{n-d})$ and continue algorithm}\label{line:output}
\Else
\State{Subdivide $B$ and add its children to $\mathcal{Q}$}
\EndIf
\EndWhile
\end{algorithmic}
\end{algorithm}

We observe that, the side length of every $B$ considered in this algorithm is at most $2$, so $s_1\leq 0$ whenever it appears.  Therefore, when approximating $\hat{y}$ in Step \ref{hat_y_approx}, the accuracy of $\hat{y}$ implies that the point $\pi_d^{-1}(\hat{x})$ is in the box constructed in the interval Krawczyk test.  Moreover, $\hat{y}$ and $\hat{A}$ in Steps \ref{hat_y_approx} and \ref{hat_A_approx}, respectively, must be approximated with enough accuracy so that the error in the computation of the interval Krawczyk test becomes negligible as the number of subdivisions increases.  Details of this requirement appear in the following correctness theorem.

\begin{theorem}\label{thm:graphcorrectness}
Let $F$, $X$, $U$, and $\rho\in(0,1)$ be as above.  
In particular,~$X$ is a graph over $U$.  If the accuracy of the approximations of $\hat{y}$ and $\hat{A}$ converge sufficiently quickly, then the set of boxes output by the GraphApproximation algorithm, Algorithm \ref{algo:Graph-approx}, deformation retracts onto $X\cap\pi^{-1}_d(U)$.
\end{theorem}
\begin{proof}
By the interval Krawczyk test, \Cref{thm:generalized_krawczyk}, each box $B\times J$ in the output of the algorithm contains the portion of $X$ over~$B$.  In addition, since $U\times\mathbb{R}^{n-d}$ deformation retracts to $X\cap\pi^{-1}_d(U)$ and the output boxes partition $U$ (ignoring overlaps on boundaries), this deformation, restricted to the output boxes, also deformation retracts onto $X\cap\pi^{-1}_d(U)$.  Therefore, to complete this correctness argument, it is enough to show that Algorithm \ref{algo:Graph-approx} terminates.

Suppose, for contradiction, that $U$ is subdivided infinitely many times and $z^\ast=(x^\ast,y^\ast)\in X$ where $x^\ast\in U$ is a point such that every box $B$ containing $x^\ast$ is subdivided.  Note that as $s_1$ and $s_2$ go to~$0$, the interval in the interval Krawczyk test converges to~$z^\ast$, see \Cref{line:intervalKrawczyk}.  
Hence, both $\square F(I,\pi_{-d}(z))$ and $(Id_{n-d}-A\square J_F(I,J))$ in Algorithm \ref{algo:Krawczyk-test} approach zero.
This yields a contradiction since
the interval Krawczyk test then succeeds
when these values are sufficiently small, i.e., when the 
box size is sufficiently small.
Therefore, there is a lower bound on the size of any box~$B$ containing $x^\ast$ which must be subdivided.  Since $U$ is a compact region, there is a global upper bound on this side length, and so all sufficiently small boxes in the queue must pass the interval Krawczyk test. Hence, the algorithm terminates, which completes the proof. 
\end{proof}

\subsection{Extensions}\label{sec:extension}
We discuss two extensions of this approach: (1) decreasing the Hausdorff distance between the approximation and the variety and (2) applying the approach where $X$ consists of several sheets.

{\em Refinement.}
For many applications, it is important to have more accurate approximations of the variety $X$.  There are a few changes that we can make 
Algorithm~\ref{algo:Graph-approx} to make the resulting collection of boxes as close to $X$ as desired.  
Since \Cref{thm:generalized_krawczyk}
provides that $\|y^\ast-\pi_{-d}(\hat{z})\|\leq r_2\rho$,
we can replace the output on Step~\ref{line:output} by $B\times (\hat{y}+\rho 2^{s_2}[-1,1]^{n-d})$.  If these boxes are still too large, we can force the regions $B$ to be smaller and/or replace $\rho$ by a smaller positive value.  

{\em Multi-sheet approximations.}
A slight weakening of the conditions above 
allows the GraphApproximation algorithm, Algorithm \ref{algo:Graph-approx}, to apply when $X$ consists of a union of $k$ graphs over $U$.  In this case, every $\pi^{-1}(x)$ consists of exactly $k$ points, one for each sheet of $X$ over $U$.  The most critical change is that, in Step~\ref{hat_y_approx}, the isolation must be completed for all $k$ points over $\hat{x}$, not just one.  This change allows each sheet of $X$ to be approximated, individually, but there is a potential global issue in that the boxes for approximations for different sheets might intersect.  To avoid this issue, we must continue to refine the subdivision until any two boxes over the same point $x\in U$ do not intersect, except possibly on their sides.  
Since the proof of \Cref{thm:graphcorrectness} extends almost directly to this case, we 
collect this result here.

\begin{corollary}\label{cor:uniongraphcorrectness}
Let $F$, $X$, $U$, and $\rho\in(0,1)$ be as above.  In particular,~$X$ is a union of graphs over $U$.  If the accuracy of the approximations of each $\hat{y}$ and $\hat{A}$ converge sufficiently quickly, then the set of boxes output by the modification to the GraphApproximation algorithm described here deformation retracts onto $X\cap\pi^{-1}_d(U)$.
\end{corollary}

\section{Discussion of limitations}\label{sec:limitations}

The approximation algorithm presented in \Cref{algo:Graph-approx}, while certified, has some weaknesses.  The following addresses 
several of them.
The removal of each weakness improves the reach of the algorithm and the quality of the output, but the changes often require more care to maintain the correctness statement.  
First, the choice of splitting the box $B$ into $2^d$ sub-boxes when the Krawczyk test fails is potentially inefficient when $d$ is large.  Since the main target for this paper consists of surfaces, where $d=2$ and $2^d=4$, we do not address the issues for higher dimensions here.

Second, since the algorithm is affected by the magnitude of the slopes appearing in $X$, the boxes resulting from \Cref{algo:Graph-approx} may have very poor aspect ratios.  This issue is particularly visible because this algorithm cannot be used to approximate closed varieties since there are always extreme points where $X$ is not regular over~$U$.  This issue cannot be immediately solved by approximating the variety $X$ over different patches and projections since one has to correctly connect the different patches.  We focus our attention on this second set of issues throughout the subsequent sections.

Third, the algorithm to find an approximation $\hat{y}$ appearing in Step \ref{hat_y_approx} of Algorithm \ref{algo:Graph-approx} is unlikely to be efficient, as stated.  Both of the suggested methods search for zeros within a given region.  In our case, the fiber $\pi_d^{-1}(\hat{x})$ is $(n-d)$-dimensional and contains a unique intersection with $X$.  We must apply multivariate root isolation techniques to iteratively larger regions until the unique zero is found.  While this works theoretically, it is expected to be quite inefficient in practice.  However, the introduction of the patching to solve the second issue above also provides an opportunity to replace this inefficient technique.

\subsection{Changing coordinate systems}

The main technique to improve \Cref{algo:Graph-approx} is to change the coordinates of the variety $X$ so that the magnitude of the slope is never large.  By performing such transformations, we can mitigate 
the scaling issues identified above.
We can make a change of coordinates using an appropriate unitary transformation.  
However, the use of such transformations make it so the proof of \Cref{thm:graphcorrectness} no longer applies.  In particular, there is no longer a single region~$U$ being partitioned.  After introducing the unitary transformation that we use
in Section~\ref{sec:Unitary}, we discuss the required changes to the topological correctness statement
in Section~\ref{sec:TopCorrect}.

\subsection{Unitary transformation}\label{sec:Unitary}

We propose to use the ideas of \Cref{sec:InitialCaseGraphs} to approximate $X$ but change the coordinates,
as needed, 
so that issues regarding large magnitudes
identified above are mitigated.  
In particular, we propose to use a unitary transformation to rotate the variety $X$ so that, locally, the tangent space is nearly parallel to the span of $d$ coordinate vectors.  This guarantees that, locally, the magnitude of the derivative is small, so the boxes with poor aspect ratios are unlikely to occur.

Suppose that $\hat{z}=(\hat{x},\hat{y})$ is a point near to the variety $X$.  Let $U\Sigma V^\ast$ be the
singular value decomposition (SVD) 
of $J_F(\hat{z})$.  Using this SVD and the change of coordinates $\tilde{z}=V^\ast z$, we define $\tilde{F}(\tilde{z})\vcentcolon=U^\ast F(V\tilde{z})$. 
Observe that, after this transformation,  
    $$
    J_{\tilde{F}}(\tilde{z})= U^* J_F(z)V=U^* U\Sigma V^* V=\Sigma.
    $$
Since there are $n-d$ polynomials defining $X$, the last $d$ columns of $\Sigma$ are zero.  Moreover, since $X$ was assumed to be smooth and $\hat{z}$ is near to $X$, the Jacobian is full rank.  Putting these observations together implies that the kernel $\ker J_{\tilde{F}}(\tilde{x})$ is spanned by $e_{n-d+1},\dots,e_n$, i.e., the last $d$ coordinate vectors.  In addition, when $z^\ast$ is a point on $X$, then $V^\ast z^\ast$ is a point on the rotated copy of $X$.  We collect these steps in the \emph{UnitaryTransformations} algorithm, Algorithm \ref{algo:transform}

\begin{algorithm}[ht]
\caption{UnitaryTransformation}
\label{algo:transform}
\begin{algorithmic}[1]
\Require  
\begin{itemize}
    \item A polynomial system $F=(f_1,\dots, f_{n-d}):\mathbb{R}^n\rightarrow\mathbb{R}^{n-d}$
    that defines a smooth variety,
    \item a point $\hat{z}\in \mathbb{R}^n$ approximating a solution to $F$.
\end{itemize}
\Ensure \begin{itemize}
    \item A transformed polynomial system $\tilde{F}$,
    \item a transformed point $\tilde{z}$, and
    \item unitary matrices $U$ and $V^*$.
\end{itemize} 
\State {Compute the SVD of $J_F(\hat{z})=U\Sigma V^*$\\
Set $\tilde{F}(z)=U^* F(Vz)$.\\
Set $\tilde{z}=V^* \hat{z}$.\\
\Return $(\tilde{F},\tilde{z}, U, V^*)$.}
\end{algorithmic}
\end{algorithm}

We note that \Cref{algo:transform} requires the computation of the SVD. For rigorous computation, we perform this computation using interval arithmetic, so $U$, $\Sigma$, and $V^*$ are all represented by interval matrices with arbitrarily small intervals.  In this case, the diagonal structure and the columns of zeros of $\Sigma$ are both maintained. Since these matrices can be made with arbitrarily small intervals, any subsequent calculation can be made as precise as needed.

\begin{remark}
At this point, there is a notational difficulty that must be addressed.  In Section \ref{sec:InitialCaseGraphs}, we based the construction on the first $d$ coordinate vectors.  On the other hand, due to the way SVD is typically defined, the key vectors are the last $d$ coordinate vectors.  Whenever we use the algorithms of Section \ref{sec:InitialCaseGraphs} on a transformed variety, we assume that the order of the coordinates has been changed so that the last $d$ coordinates from the SVD representation are treated as the first $d$ coordinates for the algorithms of Section \ref{sec:InitialCaseGraphs}.
\end{remark}

\subsection{Topological correctness}\label{sec:TopCorrect}

From this point forward, we focus on the case where $X$ is a surface, i.e., $d=2$.  While many of the ideas discussed are not dependent on dimension, it is significantly more notationally challenging to describe these operations in full generality.  We leave these details to the interested reader.

The challenge that the unitary transformation introduces is that the topological correctness proof of Theorem \ref{thm:graphcorrectness} no longer applies.  To illustrate this issue, suppose that the box $I\times J$ passes the interval Krawczyk test applied to $X$.  On the other hand, suppose that $\tilde{I}\times\tilde{J}$ is a box in transformed coordinates which passes the interval Krawczyk test on $\tilde{X}=\mathcal{V}(\tilde{F})$, i.e., the transformed version of $X$.  Putting these two boxes in the original coordinate system, 
$\tilde{I}\times\tilde{J}$ remains a box but need not be axis-aligned.  By a slight abuse of notation, we continue to use the product notation for $\tilde{I}\times\tilde{J}$, even though it is not a product in the same coordinate system as $I\times J$.  By the properties of the interval Krawczyk test, each of the boxes $I\times J$ and $\tilde{I}\times\tilde{J}$ contain a portion of $X$.  Moreover, if the boxes $I\times J$ and $\tilde{I}\times\tilde{J}$ overlap, we must decide whether the two portions of $X$ in these boxes represent, locally, the same sheet or two different sheets. 

We begin by describing the structure of $X$ within a box that passes the interval Krawczyk test.  Suppose  that the box $I\times J$ passes the interval Krawczyk test, Algorithm \ref{algo:Krawczyk-test}, applied to $X$.  In this case, $I$ is a square and $J$ is $(n-2)$-dimensional.  We call $I\times\partial J$ the \emph{top} (or \emph{bottom}) of $I\times J$, and, similarly, we call $\partial I\times J$ the \emph{sides} of $I\times J$.  There are four sides to $I\times J$, corresponding to the four sides of $I$.  We observe, by \Cref{thm:generalized_krawczyk}, that the surface $X$ cannot intersect the top of $I\times J$.  This, in turn, implies that, locally, $X$ is a graph over $I$.  In addition, this observation implies that $X$ can only leave $I\times J$ through its sides.  In particular, the intersection of $X$ with a side of $I\times J$ is a curve.

Suppose that $I\times J$ and $\tilde{I}\times\tilde{J}$ are two overlapping boxes for which the interval Krawczyk test succeeds in their own coordinate systems.  Each of these boxes includes a portion of $X$, but it is possible that $(I\times J)\cap X$ and $(\tilde{I}\times\tilde{J})\cap X$ do not belong to the same local component of $X$. We now describe a way to confirm that these boxes are capturing the same local component of $X$.  For ease of notation, we assume that $I\times J$ is given in the original coordinate system while $\tilde{I}\times\tilde{J}$ may be given in some other coordinate system.

{\em Inclusion.} Suppose that $\hat{x}\in I$ such that $\hat{x}\in\pi_d(\tilde{I}\times\tilde{J})$. The fiber $\pi_d^{-1}(\hat{x})\cap (I\times J)$ is $(n-2)$-dimensional and contains a unique point of $X$.  Moreover, by the proof of \Cref{thm:generalized_krawczyk}, the standard Krawczyk test succeeds on this slice (which itself is a lower-dimensional box).  Using the standard Krawczyk test or any other multivariate root isolation technique \cite{Yap:MirandaTest:2020,Cheng:CertifiedRealIsolation:2022}, this zero can be approximated to any desired accuracy.  Since we begin with a bounding region for the unique zero in the region, the drawbacks of multivariate root isolation, as identified above, are more limited.  Let $\hat{y}$ be the approximation to the point in $I\times J$ above $\hat{x}$ with accuracy $2^{2s_2}$, where $2^{s_2}$ is the length of a side of $J$.  The approximation $\hat{y}$ converges to the point of $X$ above $\hat{x}$ since $s_2\leq 0$.  If $\hat{y}$ is guaranteed to be within $\tilde{I}\times\tilde{J}$, even after accounting for errors, it follows that, locally, $I\times J$ and $\tilde{I}\times\tilde{J}$ approximate the same portion of the surface, which is connected through $X\cap(I\times J)\cap (\tilde{I}\times\tilde{J})$.  In fact, using the arguments of \Cref{thm:graphcorrectness}, we conclude that $(I\times J)\cup (\tilde{I}\times \tilde{J})$ deformation retracts onto $X\cap((I\times J)\cup (\tilde{I}\times \tilde{J}))$.  The algorithm for this test appears in Algorithm \ref{algo:inclusion}.  We note that a return value of \textsc{False} in Algorithm \ref{algo:inclusion} does not indicate that the surfaces do not intersect, only that the test was inconclusive.

\begin{algorithm}[ht]
\caption{InclusionTest}
\label{algo:inclusion}
\begin{algorithmic}[1]
\Require  
\begin{itemize}
    \item A polynomial system $F=(f_1,\dots, f_{n-d}):\mathbb{R}^n\rightarrow\mathbb{R}^{n-d}$ 
    that defines a smooth surface and
    \item two intersecting boxes $I\times J$ and $\tilde{I}\times\tilde{J}$ (in their own coordinate systems) with $I$ and $\tilde{I}$ two-dimensional squares which pass the interval Krawczyk test.
    \end{itemize}
\Ensure A boolean, where \textsc{True} indicates the same sheet of $X$ is in both boxes.
\State Set $\hat{x}$ to be a point in $I$ such that $\pi_d^{-1}(\hat{x})$ intersects $\tilde{I}\times\tilde{J}$.  
\State Set $2^{s_2+1}$ to be the side length of $J$
\State Set $\hat{y}$ to be an approximation to the zero of $F$ over $\hat{x}$ in $J$ of accuracy at least $2^{2s_2}$.
\If{$\hat{x}\times(\hat{y}+2^{2s_2}[-1,1]^{n-2})$ is contained in $\tilde{I}\times\tilde{J}$}
\State \Return \textsc{True}
\Else
\State \Return \textsc{False}
\EndIf
\end{algorithmic}
\end{algorithm}

{\em Exclusion.} If the inclusion test fails, we cannot immediately guarantee that the boxes $I\times J$ and $\tilde{I}\times\tilde{J}$ are approximating different portions of the surface, we may have merely chosen poorly with $\hat{x}$.  When the inclusion test fails, we attempt an exclusion test to see if we can confirm that $I\times J$ and $\tilde{I}\times\tilde{J}$ do not represent the same sheet of $X$, locally.  We first observe that within $I\times J$ and $\tilde{I}\times\tilde{J}$, $X$ is a graph over $I$ and $\tilde{I}$, respectively.  Applying the refinement of Section \ref{sec:extension} to each of $I$ and $\tilde{I}$, in their coordinate systems, we arrive at a collection of boxes that contain the surface.  By intersecting with $J$ and $\tilde{J}$, we can freely assume that these boxes are contained within $I\times J$ or $\tilde{I}\times\tilde{J}$.  If none of these output boxes intersect, then we may conclude that the original sheets of $X$ contained in $I\times J$ and $\tilde{I}\times\tilde{J}$ are distinct, locally.  We note that since $I\times J$ and $\tilde{I}\times\tilde{J}$ both contain the surface and have been rotated using Algorithm \ref{algo:transform}, both the second and third issues identified above have been mitigated.

If both of the inclusion and exclusion tests fail, we return to the inclusion test.  For each box of the refinement of $I\times J$, we apply the inclusion test.  If any of these inclusion tests succeed, we immediately conclude that $I\times J$ and $\tilde{I}\times\tilde{J}$ contain the same sheet of $X$, locally.  If none of the inclusion tests succeed, we apply the exclusion test on a more accurate refinement.  We continue alternating between these two tests until either the inclusion or exclusion test succeeds.  The complete algorithm for checking whether two boxes contain the same sheet of $X$, locally, appears in Algorithm \ref{algo:component}.  This argument proves the correctness part of the following theorem.

\begin{algorithm}[ht]
\caption{ComponentTest}
\label{algo:component}
\begin{algorithmic}[1]
\Require  
\begin{itemize}
    \item A polynomial system $F=(f_1,\dots, f_{n-d}):\mathbb{R}^n\rightarrow\mathbb{R}^{n-d}$ 
    that defines a smooth surface and
    \item two intersecting boxes $I\times J$ and $\tilde{I}\times\tilde{J}$ (in their own coordinate systems) with $I$ and $\tilde{I}$ two-dimensional squares which pass the interval Krawczyk test for some
    \item $\rho\in(0,1)$
    \end{itemize}
\Ensure 
\begin{itemize}
    \item A boolean, representing if the boxes intersect and
    \item refinements of $I\times J$ and $\tilde{I}\times\tilde{J}$
\end{itemize}
\State Set $R=\{I\times J\}$, a refinement of $I\times J$.
\State Set $\tilde{R}=\{\tilde{I}\times\tilde{J}\}$, a refinement of $\tilde{I}\times\tilde{J}$.
\Loop
\ForAll{$A\times B\in R$}
\If{InclusionTest$(A\times B,\tilde{I}\times\tilde{J})$=\textsc{True}}
\State \Return \textsc{True}, $R$, and $\tilde{R}$.
\EndIf
\EndFor
\State Set $\rho=\rho/2$.
\State Set $R$ to be a refinement of $I\times J$ using $\rho$ and requiring smaller bases in Algorithm \ref{algo:Graph-approx}
\State Set $\tilde{R}$ to be a refinement of $\tilde{I}\times\tilde{J}$ using $\rho$ and requiring smaller bases in Algorithm \ref{algo:Graph-approx}
\If{{\bf all} pairs $A\times B\in R$ and $\tilde{A}\times\tilde{B}\in\tilde{R}$ do not intersect}
\State \Return \textsc{False}, $R$, and $\tilde{R}$.
\EndIf
\EndLoop
\end{algorithmic}
\end{algorithm}

\begin{theorem}
Let $F$, $X$, $I\times J$ and $\tilde{I}\times\tilde{J}$ be as above.  Algorithm~\ref{algo:component} terminates where \textsc{True} indicates that the same sheet of $X$ is in both boxes, locally, and \textsc{False} indicates that the two sheets of $X$ in the boxes are not the same, locally.
\end{theorem}
\begin{proof}
The correctness is argued above, so we focus on termination.  If the portions of $X$ in $I\times J$ and in $\tilde{I}\times\tilde{J}$ do not correspond to the same sheet, locally, then there is a positive minimum distance between the two sheets.  Once the accuracy of the refinements is less than half this minimum distance, none of the boxes in the refinement can intersect, and the algorithm terminates.  On the other hand, if the two portions of $X$ are the same sheet, locally, then, after sufficient refinement, there is some box of the refinement that is completely contained within $\tilde{I}\times\tilde{J}$.  This box passes the inclusion test since the $\hat{y}$ is in $\tilde{I}\times\tilde{J}$, even accounting for errors.    
\end{proof}

\begin{remark}
Whenever Algorithm \ref{algo:component} returns \textsc{False}, we must replace $I\times J$ and $\tilde{I}\times\tilde{J}$ with their refinements to maintain the topological correctness.  On the other hand, if the algorithm returns \textsc{True}, it is a choice whether to replace $I\times J$ and $\tilde{I}\times\tilde{J}$ with their refinements, i.e., the replacement is not required.  We note that even if the replacements are used, any new intersections between the smaller boxes do not need to be checked to confirm that they correspond to the same sheets of $X$, it is automatic from the construction.
\end{remark}

\section{Surface approximation}\label{sec:MainAlg}
We now describe the main procedure for constructing a certified approximation of a surface by patching together local certified enclosures.  We start with the setup of Section \ref{sec:IntervalKrawczyk}.  In particular, let $F=(f_1,\dots,f_{n-2})$ define a smooth surface $X=\mathcal{V}(F)\subset\mathbb{R}^n$.  Fix a positive real number $\rho<1$.  Since non-compact surfaces cannot be covered by finitely many boxes, if $X$ is not compact, we restrict our attention to a user-specified box $D\subseteq\mathbb{R}^n$, and cover $X\cap D$.  Let $\hat{z}\in\mathbb{R}^n$ be an initial point that approximates a point of $X$ in the sense that we may improve the approximation $\hat{z}$, if needed.

The main algorithm proceeds by constructing and patching together local certified enclosures of $X$.  Concretely, starting at $\hat{z}$, we construct a box around $X$ which satisfies the interval Krawczyk test.  From the sides of this box, we construct additional boxes, growing the approximation outward.  This construction continues iteratively until the entire surface is covered.  Once the approximation is complete, we apply a post-processing step to guarantee topological properties of the output.

\subsection{Constructing boxes}

The algorithm maintains a queue of active boxes, for which extension is possible.  The main step in adding a box is, given a point $\hat{z}$, construct a box around $\hat{z}$ that satisfies the Krawczyk test.  We give our approach in the \emph{AddBox} algorithm.

\begin{algorithm}[ht]
\caption{AddBox}
\label{algo:addbox}
\begin{algorithmic}[1]
\Require  
\begin{itemize}
    \item A polynomial system $F=(f_1,\dots, f_{n-d}):\mathbb{R}^n\rightarrow\mathbb{R}^{n-d}$ that defines a smooth surface,
    \item a point $\hat{z}$ that can be refined to a point on $X$, 
    \item a positive real number $r$, and 
    \item a positive real number $\rho$.
    \end{itemize}
\Ensure A box containing $\hat{z}$ that satisfies the Krawczyk test.
\Loop
\State Set $(F,\hat{z},U,V^*)=\text{UnitaryTransformation}(F,\hat{z})$.
\State Replace $\hat{z}$ by an improved approximation within $r/2$ of the point of convergence, $z^\ast$.
\State Set $\hat{A}$ to approximate the inverse of the Jacobian at $\hat{z}$ as in Remark \ref{rem:Achoice}.
\If{IntervalKrawczykTest$(F,\hat{z},(r,r),\hat{A},\rho)$=\textsc{True}}
\State \Return $\hat{z}+r[-1,1]^n$
\EndIf
\State Set $r$ to be $r/2$
\State Improve the approximation of $\hat{z}$ to an improved approximation within $r/2$ of the point of convergence, $z^\ast$.
\EndLoop
\State \Return final box
\end{algorithmic}
\end{algorithm}

By a similar argument as the termination of Algorithm \ref{algo:Graph-approx}, when $\hat{z}$ and $\hat{A}$ are sufficiently accurate, the interval Krawczyk test succeeds.  We note that since we have applied a unitary transformation and can improve the approximation of $\hat{z}$ arbitrarily well, all terms in the Krawczyk test approach zero and Algorithm \ref{algo:addbox} eventually succeeds.

The initial step of our algorithm is to use the \Cref{algo:addbox} to find a box containing $\hat{z}$ and satisfying the Krawczyk test.  The algorithm then maintains a queue consisting of pairs of boxes of the form $I\times J$ and subsets of $\partial I$.  The box $I\times J$ is the box itself (in its coordinate system) and the subset of $\partial I$ represents the how much of the curve $X\cap (\partial I\times J)$ remains to be covered by neighboring boxes.  Once all of $\partial I$ has been covered, the box $I$ is removed from the queue.

In each iteration of the algorithm, a box $I\times J$ and a subset $A\subseteq \partial I$ are popped from the queue.  A point $\hat{x}\in A$ is chosen, and its corresponding preimage $\pi_d^{-1}(\hat{z})\cap (I\times J)$ is approximated by $\hat{y}$.  Then a box $\tilde{I}\times \tilde{J}$ is constructed containing $\hat{z}=(\hat{x},\hat{y})$.  Finally, for all the boxes that intersect $\tilde{I}\times\tilde{J}$, their subsets of $\partial I$ are recalculated, and, unless $A$ is not empty, $I\times J$ and the newly updated $A$ are pushed into the queue.  In addition, a similar calculation is executed for $\tilde{I}\times\tilde{J}$ and $\partial\tilde{I}$, which are also pushed into the queue, if $\partial\tilde{I}$ is nonempty.

\subsection{Covering the surface}

We now provide details about the pairing of a box $I\times J$ and the subset of $\partial I$.  The idea of the pairing is $I\times J$ contains a portion of~$X$, and we have seen that, restricted to the sides of $I\times J$, $X$ is a curve.  On all sides of $I\times J$, we need to have enough boxes to cover this curve.  Once this curve is covered by boxes, these neighboring boxes continue the surface approximation beyond $I\times J$.  There is a bijection between this curve in the sides of $I\times J$ and $\partial I$.  The idea is that the subset of $\partial I$ keeps track of the part of this curve which we cannot yet confirm to be covered.  

We start with the setup above, a box $I\times J$, a point $\hat{z}$ in the boundary of $I\times J$ such that $\pi_d(\hat{z})\in A$, and a box $\tilde{I}\times\tilde{J}$ 
containing~$\hat{z}$ and satisfying the interval Krawczyk test.  At this point, we know that $I\times J$ and $\tilde{I}\times\tilde{J}$ approximate the same local sheet of $X$ because they share the point $\hat{z}$.  For any other boxes $I'\times J'$ that intersect $\tilde{I}\times\tilde{J}$, we apply the component test.  If the test fails, we replace $I'\times J'$ by the refinement from the test, improve the approximation on $\hat{z}$, and repeat the construction of the box $\tilde{I}\times\tilde{J}$ using $r/2$. 

Once this construction succeeds, we have two boxes both containing~$\hat{z}$ and the point that $\hat{z}$ converges to $z^\ast$.  We now apply the interval Krawczyk test within the facet of $I\times J$.  In particular, let $r$ be some value so that the distance between $\hat{z}$ and $z^\ast$ is less than $r$ and $\hat{z}+r[-1,1]^{d-1}$ is within the intersection of the facet of $I\times J$ with the box $\tilde{I}\times\tilde{J}$.  For simplicity, let $\ell$ be the equation of the facet of $I\times J$ containing $\hat{z}$.  We apply AddBox to $F\cup\{\ell\}$, $\hat{z}$, $r$, and $\rho$.  This constructs a box $B$ within the facet of $I\times J$ and within the box $\tilde{I}\times\tilde{J}$.  Moreover, by the interval Krawczyk test, for every point in $\pi_d(B)$, there is a point of $X$ in $B$ above that point.  Therefore the curve of interest passes through $B$ and this curve of interest is both on the boundary of $I\times J$ and within $\tilde I\times\tilde J$.  Therefore, we may remove~$\pi_d(B)$ from $A$ as this portion of the curve of interest is covered.

In the noncompact case, we also remove portions of $\partial I$ which lie outside the region $D$.  Since this case is easier, we leave the details to the reader. We summarize the entire construction in Algorithm~\ref{algo:CertifiedSurfaceApproximation}.

\begin{algorithm}[ht]
\caption{CertifiedSurfaceApproximation}
\label{algo:CertifiedSurfaceApproximation}
\begin{algorithmic}[1]
\Require  
\begin{itemize}
    \item A polynomial system $F=(f_1,\dots,f_{n-d}):\mathbb{R}^n\to\mathbb{R}^{n-d}$ defining a smooth surface $X=\mathcal{V}(F)$,
    \item an initial point $\hat{z}\in\mathbb{R}^n$ approximating a point of $X$,
    \item an initial positive radius $r\in\mathbb{Q}$,  
    \item a constant $\rho\in(0,1)$, and,
    \item if needed, a compact region of interest $D\subseteq\mathbb{R}^n$.
\end{itemize}
\Ensure A finite collection of boxes whose union encloses $X$ or $X\cap D$.
\State Compute a starting box $I\times J$ from AddBox$(F,\hat{z},r,\rho)$.
\State Set $A$ to be $\partial I$ minus any fibers over $\partial I$ which are completely outside of $D$.
\State Set $\mathcal{Q}=\{(I\times J,A)\}$.
\While{$\mathcal{Q}\not=\emptyset$}
\State Pop $(I\times J,A)$ from $\mathcal{Q}$.
\State Set $\hat{x}$ to be any point in $A$.
\State Set $\hat{z}$ to approximate the point $\pi_d^{-1}(\hat{x})\cap I\times J$.
\State Set $s=2r$
\Repeat
\State $s=s/2$.
\State Set $\tilde{I}\times\tilde{J}$ from AddBox$(F,\hat{z},s,\rho)$.
\ForAll{Boxes $I'\times J'$ intersecting $\tilde{I}\times\tilde{J}$}
\If{ComponentTest$(F,\tilde{I}\times\tilde{J},I'\times J')=\textsc{False}$}
\State Replace $I'\times J'$ with a refinement 
\State Update subsets of boundaries of refinement.
\EndIf
\EndFor
\Until{All ComponentTests return $\textsc{True}$}
\ForAll{Boxes $I'\times J'$ intersecting $\tilde{I}\times\tilde{J}$}
\State Update subset of boundary of $I'\times J'$
\State Update subset of boundary of $\tilde{I}\times\tilde{J}$
\EndFor
\If{Subset of boundary of $\tilde{I}\times\tilde{J}\not=\emptyset$}
\State Add $\tilde{I}\times\tilde{J}$ and its subset of boundary to queue.
\EndIf
\State Remove boxes from $\mathcal{Q}$ whose boundary subset is empty.
\EndWhile
\State \Return all final boxes produced by the algorithm.
\end{algorithmic}
\end{algorithm}

\begin{theorem}
Let $F$, $X$, $\hat{z}$, $r$, $\rho$, and $D$ (if needed) as above.  Algorithm \ref{algo:CertifiedSurfaceApproximation} terminates.
\end{theorem}
\begin{proof}
Similar to the proof of Theorem \ref{thm:graphcorrectness}, when the boxes get sufficiently small, the interval Krawczyk test succeeds.  Therefore, there is a global lower bound on the size of a box where the Krawczyk test fails.  Similarly, there is a lower bound on the distance between distinct sheets of $X$.  Hence, there is a global lower bound on the size of a box where the component test may fail.  Finally, there is an upper bound on the speed at which the tangent spaces change.  Therefore, there is a global lower bound on the boxes that occur in the construction of Algorithm~\ref{algo:CertifiedSurfaceApproximation}.  Thus, there is a lower bound on the area of boundary that is covered by each box that is added through the algorithm.  Then, by compactness, a finite number of boxes needed until the surface is covered.
\end{proof}

\subsection{Correctness}

To ensure the topological correctness of the entire construction, we can apply a few final processing steps.  For any two boxes $I\times J$ and $\tilde{I}\times\tilde{J}$ that intersect, apply ComponentTest$(F,I\times J,\tilde{I}\times\tilde{J},\rho)$ to the boxes.  If this returns \textsc{False}, we know that the part of the overlap between $I\times J$ and $\tilde{I}\times\tilde{J}$ does not include the surface.  Therefore, the intersection $(I\times J)\cap (I'\times J')$ can be removed from the approximation.  If these changes create any disconnected components of the approximation, then one of the components does not contain zeros of $F$, and can be removed. The local correctness implies that we have constructed an enclosing neighborhood of the surface such that any two intersecting boxes contain the same portion of the surface, i.e., the global structure is maintained.

\section{Implementations and examples}\label{sec:Implementation}

We provide a preliminary Julia implementation of our algorithms. For interval arithmetic, we use \texttt{IntervalArithmetic.jl} \cite{IntervalArithmetic.jl}. In this section, we present examples of certified surface approximations to illustrate the proposed algorithms. Throughout these examples, we use an initial radius of $r = 0.1$. Starting from an input point sufficiently close to the surface, we first generate an initial interval box that passes the generalized interval Krawczyk test.  Throughout the examples, we colored the initial box in red.




All examples are displayed using \texttt{MeshLab} \cite{LocalChapterEvents:ItalChap:ItalianChapConf2008:129-136}. The source code for the implementation is available at
\begin{center}
\url{https://github.com/klee669/certified_surface_approximations}
\end{center}

\begin{example}[A unit sphere]
We approximate the unit sphere defined by $x^2 + y^2 + z^2 - 1 = 0$, starting from the initial point $(0,0,1)$ and using $\rho = \frac{1}{8}$. The algorithm terminates with $5975$ certified interval boxes. The resulting approximation is shown in \Cref{fig:unit_sphere}.

\begin{figure}[t]
    \centering
    \includegraphics[scale=.15]{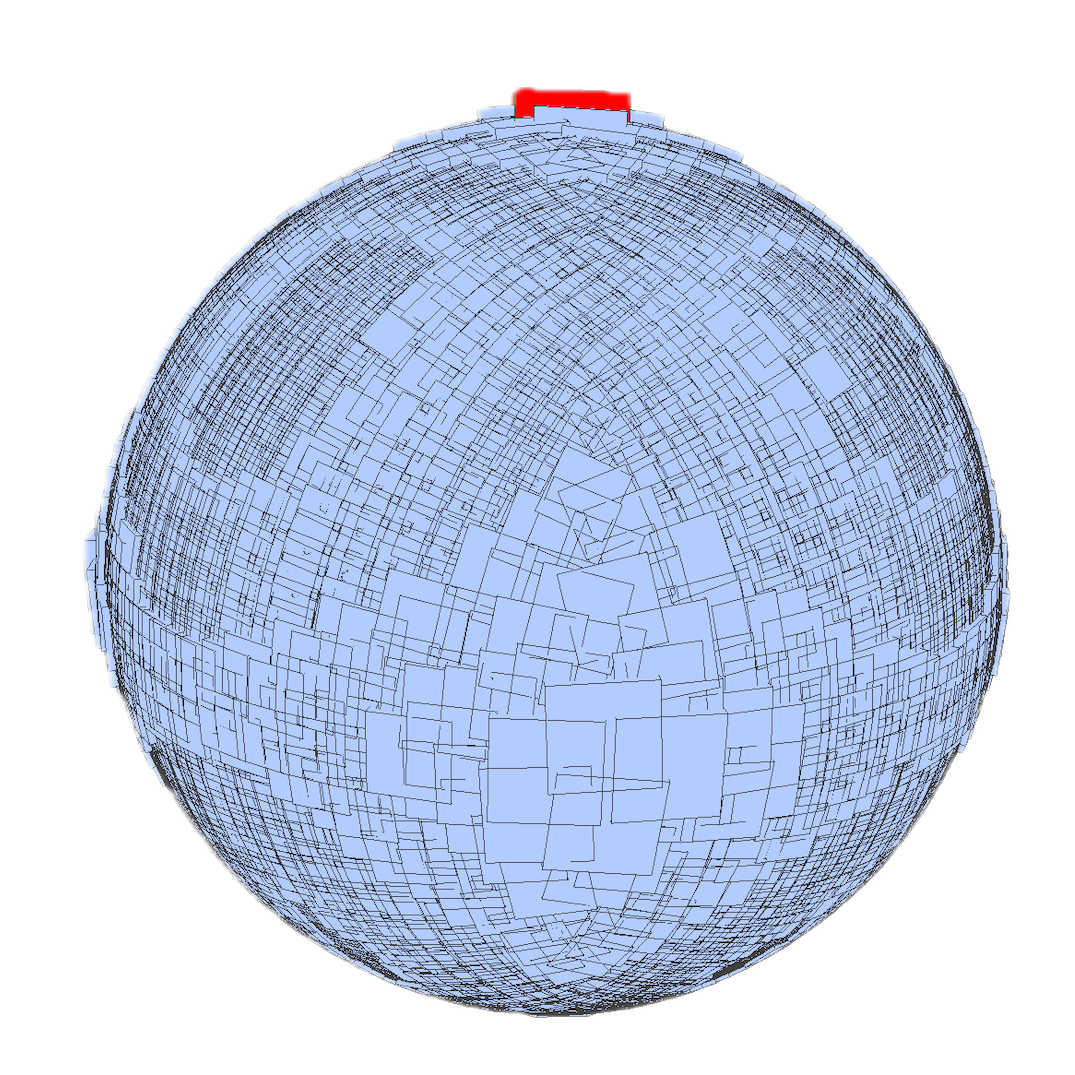}
    \caption{An approximation of a unit sphere $\bm{x^2+y^2+z^2-1=0}$ with an initial point $\bm{\hat{z}=(0,0,1)}$ and $\bm{\rho=\frac{1}{8}}$.}
    \label{fig:unit_sphere}
\end{figure}

\end{example}

\begin{example}[Effect of refinement]
To illustrate the effect of refinement, we approximate the sphere $x^2 + y^2 + z^2 - 10 = 0$ using two different values of $\rho$, namely $\rho = \frac{1}{8}$ and $\rho = \frac{1}{80}$ at an initial point $(0,0, 3.16)$. To better visualize the coverage of the surface, we limit the computation to at most $2000$ interval boxes in both cases.

For $\rho = \frac{1}{8}$, the computation terminates with an average box radius of approximately $0.1453$. For $\rho = \frac{1}{80}$, the average box radius is approximately $0.02510$. As expected, a smaller value of $\rho$ yields a finer approximation, see \Cref{fig:sphere_comparison1}. 

For comparison, \Cref{fig:sphere_comparison2} shows the approximation with $\rho=\frac{1}{8}$ at $(0,0,3.16)$, and that with $\rho=\frac{1}{80}$ at $(0,0,-3.16)$ at once.

\begin{figure}[t]
    \centering
    \includegraphics[scale=.08]{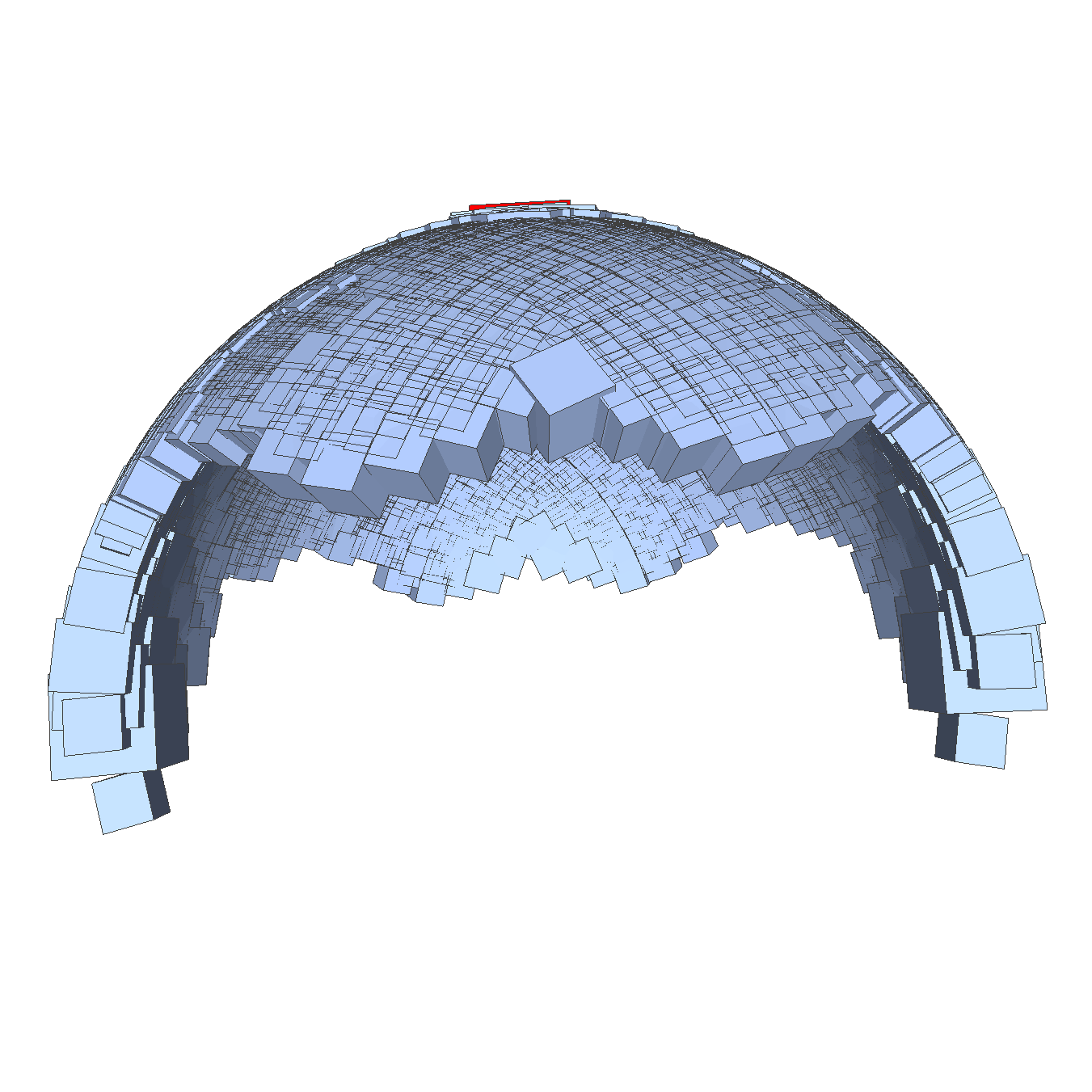}
    \hspace{.5em}
    \includegraphics[scale=.08]{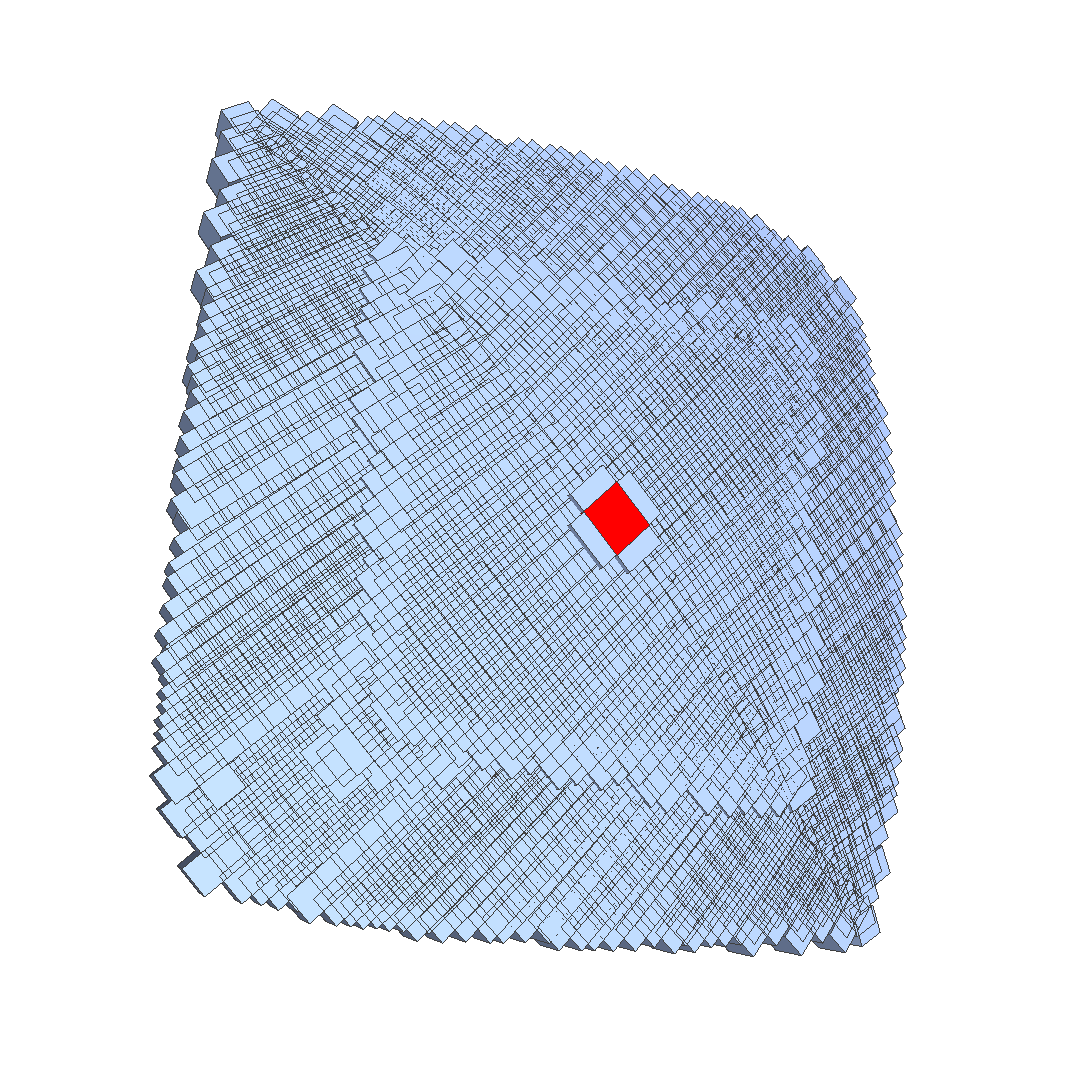}
    \caption{Approximations of a sphere $\bm{x^2+y^2+z^2-10=0}$ with different $\bm{\rho=\frac{1}{8}}$ (left) and $\bm{\frac{1}{80}}$ (right) at an initial point $\bm{\hat{z}=(0,0,3.16)}$. Each approximation shows $\bm{2000}$ interval boxes respectively.}
    \label{fig:sphere_comparison1}
\end{figure}
\begin{figure}[t]
    \centering
    \includegraphics[scale=.15]{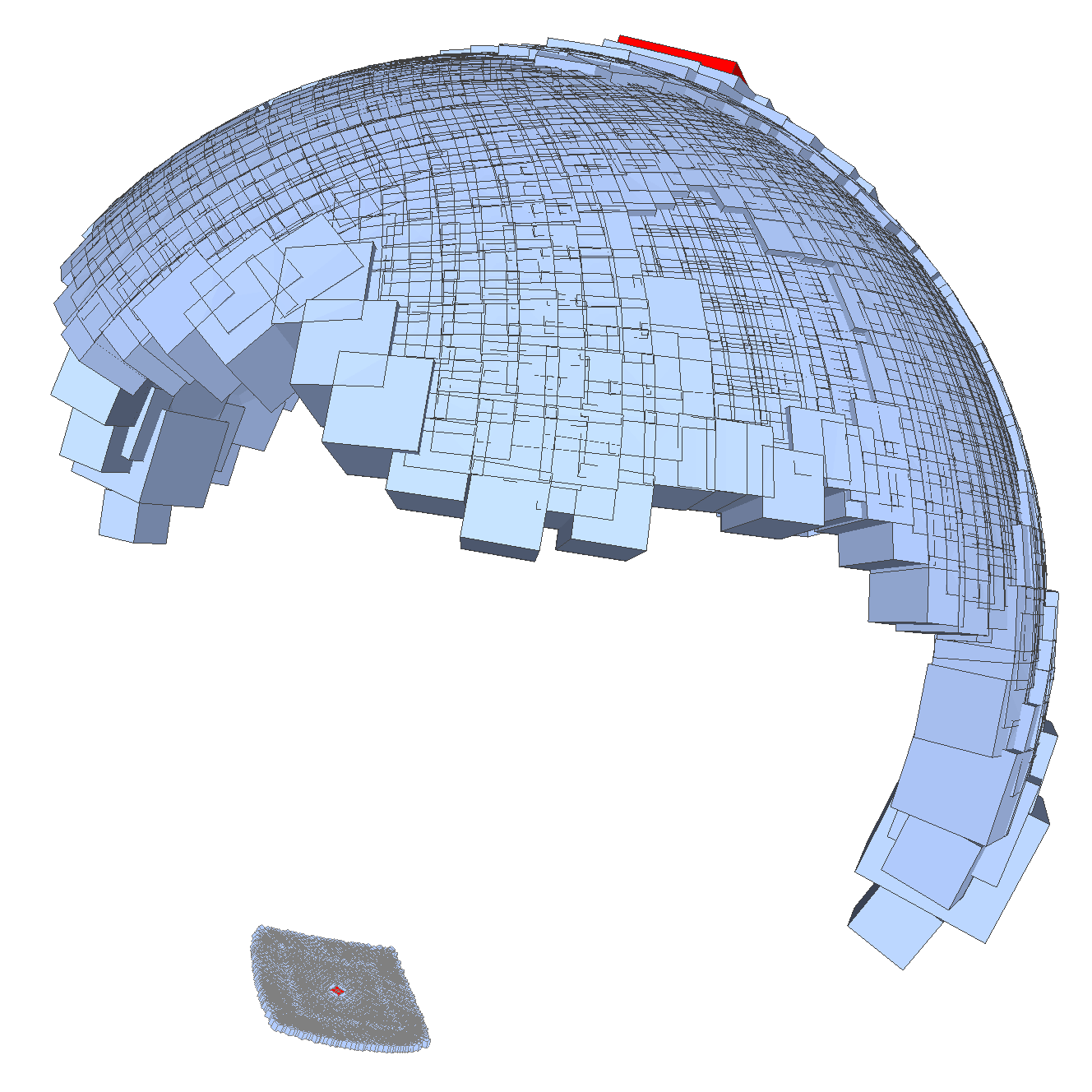}
    \caption{Approximations of a sphere $\bm{x^2+y^2+z^2-10=0}$ with different $\bm{\rho=\frac{1}{8}}$ at $\bm{\hat{z}=(0,0,3.16)}$ and with $\bm{\frac{1}{80}}$ at an initial point $\bm{\hat{z}=(0,0,-3.16)}$. Each approximation shows $\bm{2000}$ interval boxes respectively.}
    \label{fig:sphere_comparison2}
\end{figure}

\end{example}

\begin{example}[Torus]
    We approximate the torus defined by an equation $(\sqrt{x^2 + y^2} - 2)^2 + z^2 - 0.64 = 0$, starting from the initial point $(2.8,0,0)$. We use $\rho = \frac{7}{8}$, and the algorithm produces an approximation consisting of $2400$ certified interval boxes.  The resulting approximation is in \Cref{fig:torus}.  Although not visible in the image, the interior of the torus is a tube, as expected.

\end{example}

\begin{example}[Saddle surface]
As a non-compact example, we approximate the saddle surface defined by $-0.125 x y^2 + 0.25 x^2 - z = 0$
starting from the initial point $(2,2,0)$. We use $\rho = \frac{7}{8}$. Since the surface is non-compact, we restrict the computation to $2400$ certified interval boxes.  See \Cref{fig:saddle}.
    
\end{example}

\begin{figure}[h]
    \centering
    \includegraphics[scale=.15]{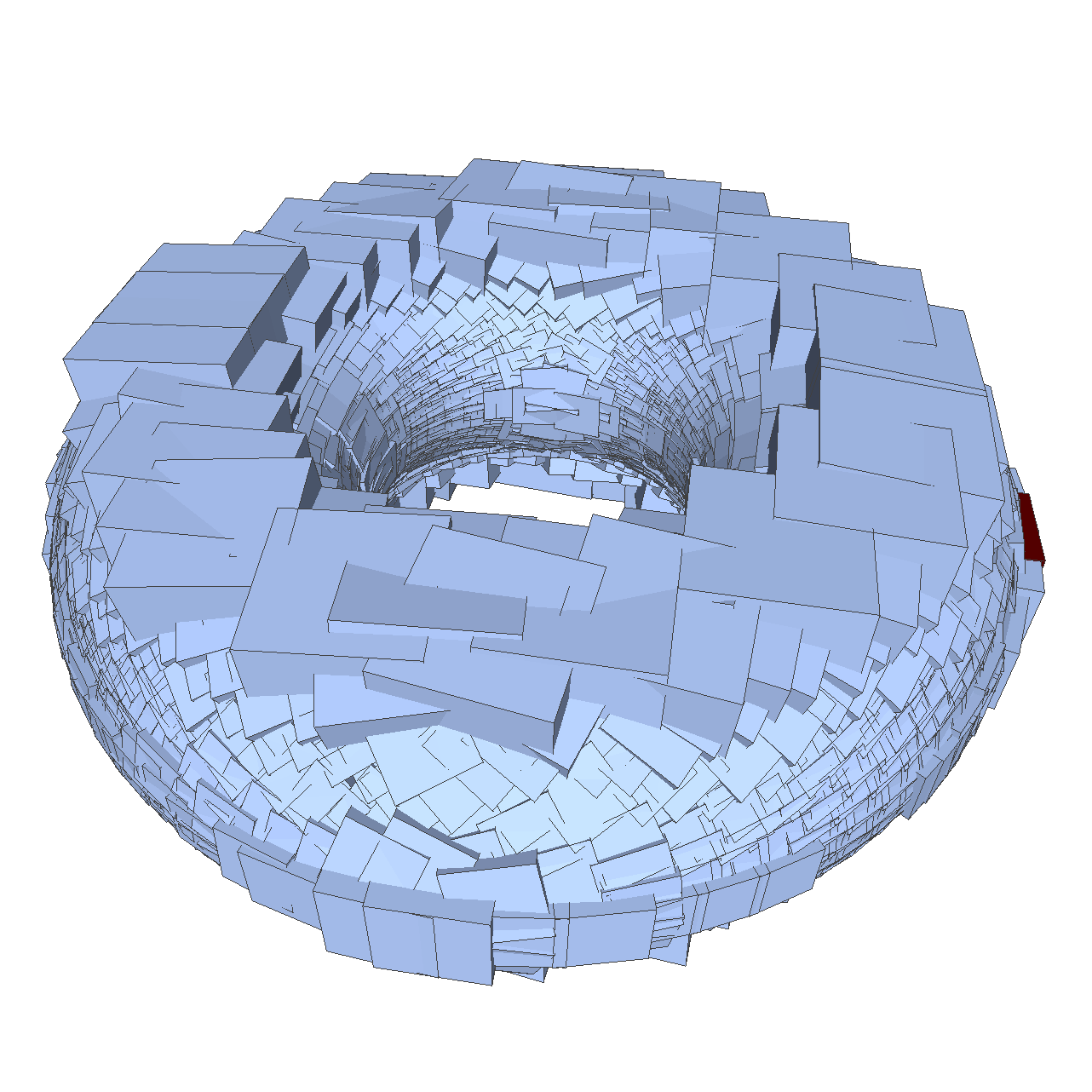}
    \caption{An approximation of a torus $\bm{(\sqrt{x^2 + y^2} - 2)^2 + z^2 - 0.64 = 0}$ with an initial point $\bm{\hat{z}=(2.8,0,0)}$ and $\bm{\rho=\frac{7}{8}}$.}
    \label{fig:torus}
\end{figure}

\begin{figure}[h]
    \centering
    \includegraphics[scale=.15]{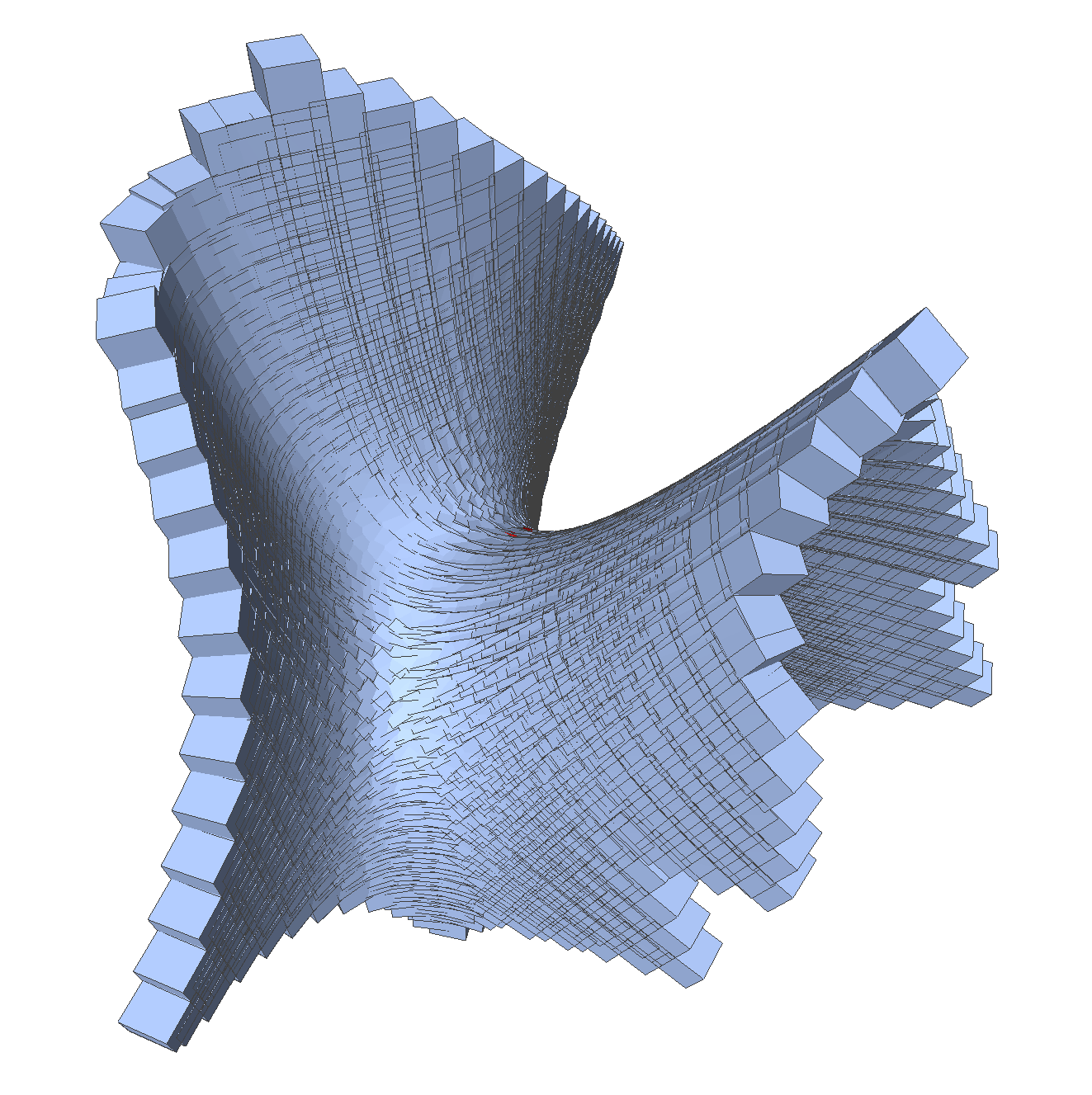}
    \caption{An approximation of a saddle surface $\bm{-0.125 x y^2 + 0.25 x^2 - z = 0}$ with an initial point $\bm{\hat{z}=(2,2,0)}$ and $\bm{\rho=\frac{7}{8}}$.}
    \label{fig:saddle}
\end{figure}
\bibliography{ref}

\begin{thebibliography}{10}

\bibitem{beltran2012certified}
C.~Beltr{\'a}n and A.~Leykin.
\newblock Certified numerical homotopy tracking.
\newblock {\em Experimental Mathematics}, 21(1):69--83, 2012.

\bibitem{burr2025certified}
M.~Burr, M.~Byrd, and K.~Lee.
\newblock Certified algebraic curve projections by path tracking.
\newblock In {\em Proceedings of the 2025 International Symposium on Symbolic
  and Algebraic Computation}, pages 87--96, 2025.

\bibitem{Cheng:CertifiedRealIsolation:2022}
J.-S. Cheng and J.~Wen.
\newblock Certified numerical real root isolation of zero-dimensional
  multivariate real nonlinear systems.
\newblock Technical Report arXiv:2211.05266 [cs.CG], 2022.

\bibitem{LocalChapterEvents:ItalChap:ItalianChapConf2008:129-136}
P.~Cignoni, M.~Callieri, M.~Corsini, M.~Dellepiane, F.~Ganovelli, and
  G.~Ranzuglia.
\newblock {MeshLab: an Open-Source Mesh Processing Tool}.
\newblock In V.~Scarano, R.~D. Chiara, and U.~Erra, editors, {\em Eurographics
  Italian Chapter Conference}. The Eurographics Association, 2008.

\bibitem{duff2024certified}
T.~Duff and K.~Lee.
\newblock Certified homotopy tracking using the {K}rawczyk method.
\newblock In {\em Proceedings of the 2024 International Symposium on Symbolic
  and Algebraic Computation}, pages 274--282, 2024.

\bibitem{guillemot2024validated}
A.~Guillemot and P.~Lairez.
\newblock Validated numerics for algebraic path tracking.
\newblock In {\em Proceedings of the 2024 International Symposium on Symbolic
  and Algebraic Computation}, pages 36--45, 2024.

\bibitem{Hauenstein:2016}
J.~D. Hauenstein and A.~C. Liddell.
\newblock Certified predictor–corrector tracking for {N}ewton homotopies.
\newblock {\em Journal of Symbolic Computation}, 74:239--254, 2016.

\bibitem{krawczyk1969newton}
R.~Krawczyk.
\newblock Newton-{A}lgorithmen zur {B}estimmung von {N}ullstellen mit
  {F}ehlerschrank.
\newblock {\em Computing}, 4:187--201, 1969.

\bibitem{Yap:MirandaTest:2020}
J.-M. Lien, V.~Sharma, G.~Vegter, and C.~Yap.
\newblock Isotopic arrangement of simple curves: An exact numerical approach
  based on subdivision.
\newblock In {\em International Congress on Mathematical Software}, pages
  277--282. Springer, 2014.

\bibitem{moore2009introduction}
R.~E. Moore, R.~B. Kearfott, and M.~J. Cloud.
\newblock {\em Introduction to interval analysis}.
\newblock SIAM, 2009.

\bibitem{rump1983solving}
S.~M. Rump.
\newblock Solving algebraic problems with high accuracy.
\newblock In {\em A new approach to scientific computation}, pages 51--120.
  Elsevier, 1983.

\bibitem{IntervalArithmetic.jl}
D.~P. Sanders and L.~Benet.
\newblock Interval{A}rithmetic.jl, 2014.

\bibitem{Smale:1993}
M.~Shub and S.~Smale.
\newblock Complexity of {B}ezout’s theorem {I}: Geometric aspects.
\newblock {\em Journal of American Mathematical Society}, 6(2):459--501, 1993.

\bibitem{SmaleAlphaTheory}
S.~Smale.
\newblock Newton's method estimates from data at one point.
\newblock In {\em The merging of disciplines: new directions in pure, applied,
  and computational mathematics ({L}aramie, {W}yo., 1985)}, pages 185--196.
  Springer, New York, 1986.

\end{thebibliography}
\bibliographystyle{abbrv}

\end{document}